\documentclass[10pt]{amsart}

\usepackage[all]{xy}
\usepackage{amsmath}
\usepackage{mathrsfs}
\usepackage{amssymb}
\usepackage{mathbbol}
\usepackage{a4wide}

\usepackage[utf8]{inputenc}
\usepackage[T1]{fontenc}

\newtheorem{theorem}{Theorem}[section]

\theoremstyle{definition}
\newtheorem{definition}[theorem]{Definition}
\newtheorem{example}[theorem]{Example}

\newcommand{\cst}{\mathrm{C}^*}
\newcommand{\CC}{\mathbb{C}}
\newcommand{\RR}{\mathbb{R}}
\newcommand{\TT}{\mathbb{T}}
\newcommand{\ZZ}{\mathbb{Z}}
\newcommand{\GG}{\mathbb{G}}
\newcommand{\sS}{\mathbb{S}}
\newcommand{\st}{\:\vline\:}
\renewcommand{\Bar}[1]{\overline{#1}}
\newcommand{\id}{\mathrm{id}}
\newcommand{\tens}{\otimes}
\newcommand{\I}{\mathbb{1}}
\newcommand{\comp}{\!\circ\!}
\newcommand{\cA}{\mathscr{A}}
\newcommand{\cK}{\mathscr{K}}
\newcommand{\cJ}{\mathscr{J}}
\newcommand{\cI}{\mathscr{I}}
\newcommand{\atens}{\tens_{\textrm{\tiny alg}}}
\DeclareMathOperator{\C}{C}
\newcommand{\cT}{\mathcal{T}}
\newcommand{\bz}{\boldsymbol{z}}
\newcommand{\bx}{\boldsymbol{x}}
\begin{document}

\title{When is a quantum space not a group?}

\author{Piotr M.~So{\l}tan}
\address{Institute of Mathematics of the Polish Academy of Sciences\\
and\\
Department of Mathematical Methods in Physics, Faculty of Physics, University of Warsaw}
\email{{\tt piotr.soltan@fuw.edu.pl}}

\thanks{Research partially supported by Polish government grant no.~N201 1770 33.}

\begin{abstract}
We give a survey of techniques from quantum group theory which can be used to show that some quantum spaces (objects of the category dual to the category of $\mathrm{C}^*$-algebras) do not admit any quantum group structure. We also provide a number of examples which include some very well known quantum spaces. Our tools include several purely quantum group theoretical results as well as study of existence of characters and traces on $\mathrm{C}^*$-algebras describing the considered quantum spaces as well as properties such as nuclearity.
\end{abstract}

\maketitle

\hyphenation{non-emp-ty}

\section{Introduction}\label{intro}

Let $X$ be a topological space. One can easily turn $X$ into an associative topological semigroup. A possible definition of multiplication is $x\cdot{y}=y$ for all $x,y\in{X}$. This is clearly not a group (unless $X$ consists of a single point). One is therefore lead to a more refined question whether $X$ can be turned into a topological \emph{group}.

It is not difficult to find topological spaces which cannot be given a structure of a topological group at all. As an example of this phenomenon consider the interval $[0,1]$. Clearly this topological space cannot be a topological group because there is no homeomorphism of $[0,1]$ onto itself carrying an end point onto an interior point (the endpoints have neighborhoods which are connected after the endpoint is removed and interior points do not have such neighborhoods) and there would have been one if $[0,1]$ were a topological group. The same argument shows that no manifold with boundary can be endowed with a structure of a topological group.

One is forced apply much more sophisticated tools to prove, for example, that the two-sphere $\sS^2$ is not a topological group with its usual topology. Of course, the arguments produced above no longer work, as $\sS^2$ is a homogeneous space. Still, using some results about vector bundles or by noticing that the cohomology ring of $\sS^2$ does not admit a Hopf algebra structure, the conclusion that $\sS^2$ is not a group can be reached (\cite[Section 3.C]{hatch}).

In this paper we want to address similar questions, but instead of topological spaces we want to consider objects of \emph{noncommutative topology} or \emph{quantum spaces,} i.e.~objects of the category dual to the category of $\cst$-algebras (\cite{connes,unbo}).
Perhaps not unexpectedly the tools we will use to show that some well known quantum spaces do not admit a group structure (they are not \emph{quantum groups} as defined in Section \ref{qgsect}) are of completely different nature than those known from classical topology.

Let us briefly describe the contents of the paper. In Section \ref{qgsect} we give definitions of objects of our study such as quantum spaces, quantum semigroups and compact quantum groups. We also provide some basic examples and introduce standard terminology. Section \ref{addi} is devoted to a survey of results from the theory of compact quantum groups which we need in following sections. In particular we describe the constructions of the \emph{reduced} and \emph{universal} versions of a given compact quantum group and define \emph{Woronowicz characters}. The next section introduces the tools used to show that some well known quantum spaces do not admit a compact quantum group structure. The results are tailored to suit applications and proofs of some known facts (e.g.~from \cite{bmt}) are considerably simplified. Finally, in Section \ref{ex} we describe in detail examples of quantum spaces not admitting any compact quantum group structure. 
At the end of this section we discuss some partial results about the \emph{quantum disk}.

\section{Compact quantum semigroups and compact quantum groups}\label{qgsect}

Let us first define \emph{compact quantum spaces.} The category of compact quantum spaces is by definition the category dual to the category of unital $\cst$-algebras with unital $*$-homomorphisms as morphisms (the compactness of our quantum spaces is encoded in the fact that all considered $\cst$-algebras will have a unit). The choice to restrict attention to compact quantum spaces is motivated mainly by the amazingly rich theory of \emph{compact quantum groups} as defined below.

\begin{definition}\label{Defqg}
\noindent\begin{enumerate}
\item\label{Defqg1} A \emph{compact quantum semigroup} is a pair $\GG=(A,\Delta)$, where $A$ is a unital $\cst$-algebra and $\Delta$ is a morphism $A\to{A}\tens{A}$ (minimal tensor product of $\cst$-algebras) such that
\[
(\id\tens\Delta)\comp\Delta=(\Delta\tens\id)\comp\Delta.
\]
The morphism $\Delta$ is called the \emph{comultiplication} of $\GG$.
\item\label{Defqg2} A \emph{compact quantum group} is a compact quantum semigroup $\GG=(A,\Delta)$ such that the linear spans of the sets
\[
\bigl\{\Delta(a)(\I\tens{b})\st{a,b}\in{A}\bigr\}\quad\text{and}\quad\bigl\{(a\tens\I)\Delta(b)\st{a,b}\in{A}\bigr\}
\]
are dense in $A\tens{A}$.
\end{enumerate}
\end{definition}

\begin{example}\label{com}
Let $G$ be a compact associative semigroup and let $A=\C(G)$. We define $\Delta:A\to{A}\tens{A}=\C(G\times{G})$ by
\[
(\Delta{f})(s,t)=f(st).
\]
Then $\GG=(A,\Delta)$ is a compact quantum semigroup as in Definition \ref{Defqg}\eqref{Defqg1}. The density conditions from Definition \ref{Defqg}\eqref{Defqg2} are equivalent to cancellation laws:
\begin{equation}\label{cancel}
\begin{split}
\bigl(s\cdot{t}=s'\cdot{t}\bigr)&\Longrightarrow\bigl(s=s'\bigr),\\
\bigl(s\cdot{t}=s\cdot{t'}\bigr)&\Longrightarrow\bigl(t=t'\bigr),
\end{split}
\end{equation}
so $\GG=(A,\Delta)$ is a compact quantum group if and only id the implications \eqref{cancel} hold for any $s,s',t,t'\in{G}$, i.e.~precisely when $G$ is a compact group.
\end{example}

It is a fact that any compact quantum group $\GG=(A,\Delta)$ with $A$ commutative is necessarily of the form described in Example \ref{com}.

\begin{example}\label{cocom}
Let $\Gamma$ be a discrete group and let $A=\cst(\Gamma)$. Since $A$ is the completion of $\ell^1(\Gamma)$ in the maximal $\cst$-norm, there is a copy of $\Gamma$ inside the unitary group of $A$. Moreover $A$ is generated by these elements. One can easily see that there exists a unique $\Delta:A\to{A}\tens{A}$ such that
\[
\Delta(\gamma)=\gamma\tens\gamma
\]
for all $\gamma\in\Gamma$. It is clear that $\GG=(A,\Delta)$ is a compact quantum group.
\end{example}

The comultiplication of the compact quantum group $\GG=(A,\Delta)$ described in Example \ref{cocom} is \emph{cocommutative,} i.e.~$\Delta=\sigma\comp\Delta$, (where $\sigma:{A}\tens{A}\to{A}\tens{A}$ is the flip). One is tempted to write that all cocommutative compact quantum groups are of the form described in Example \ref{cocom}. However this statement is not true. It is true for \emph{universal} compact quantum groups which we will define in Subsection \ref{univ}.

There are numerous other examples of compact quantum groups in the literature for which we refer the reader to e.g.~\cite{su2,wvd,qdt}.

\section{Additional structure}\label{addi}

Throughout this section we let $\GG=(A,\Delta)$ be a compact quantum group. The $\cst$-algebra $A$ carries very rich additional structure. As we will see in Section \ref{ex}, this fact may be used to decide whether a given compact quantum space can be endowed with a compact quantum group structure.

The results about compact quantum groups formulated in this section are all covered in \cite{cqg} except for the material of Subsection \ref{univ} which can be found in \cite[Section 3]{bmt}.

\subsection{Hopf algebra}

The first result we want to state relates compact quantum groups to Hopf $*$-algebras. For the theory of Hopf algebras we refer to \cite{sweedler}. Hopf $*$-algebras and their generalizations are discussed e.g.~in \cite{podmu}.

\begin{definition}
An $n$-dimensional unitary representation of $\GG$ is a unitary matrix
\[
u=
\begin{bmatrix}
u_{1,1}&\cdots&u_{1,n}\\
\vdots&\ddots&\vdots\\
u_{n,1}&\cdots&u_{n,n}
\end{bmatrix}
\in{M_n(\CC)}\tens{A}=M_n(A)
\]
such that
\begin{equation}\label{Delu}
\Delta(u_{i,j})=\sum_{k=1}^nu_{i,k}\tens{u_{k,j}}.
\end{equation}
Elements $\{u_{i,j}\}_{i,j=1,\ldots,n}$ are called the \emph{matrix elements of $u$}.\index{matrix elements of a representation}
\end{definition}

The concept of a finite dimensional unitary representation of a compact quantum group is a straightforward generalization of the notion of a finite dimensional unitary representation of a compact group. It is easy to see that in the case presented in Example \ref{com} representations of the compact quantum group are the same objects as ordinary representations of the group.

\begin{theorem}[S.L.~Woronowicz]\label{cB}
Let $\GG=(A,\Delta)$ be a compact quantum group and let $\cA$ be the linear span of matrix elements of all finite dimensional unitary representations of $\GG$. Then
\begin{enumerate}
\item $\bigl.\Delta\bigr|_{\cA}(\cA)\subset\cA\atens\cA$,
\item $\cA$ is dense unital $*$-subalgebra in $A$,
\item if $\Delta_{\cA}$ is the restriction of $\Delta$ to $\cA$ then $(\cA,\Delta_{\cA})$ is a Hopf $*$-algebra. In particular there exist antipode $\kappa:\cA\to\cA$ and counit $e:\cA\to\CC$.
\item $\cA$ is the unique Hopf $*$-algebra which can be embedded in $(A,\Delta)$ as a dense $*$-subalgebra with the same comultiplication.\footnote{The reference for the last statement of Theorem \ref{cB} is \cite[Theorem 5.1]{bmt}.}
\end{enumerate}
\end{theorem}

Let $\GG=(A,\Delta)$ be a compact quantum group. The dense Hopf $*$-algebra $\cA$ of $A$ is called the \emph{Hopf $*$-algebra associated to $\GG$.} In case of a classical group $G$ (as in Example \ref{com}) the associated Hopf $*$-algebra is the algebra of regular functions on $G$. In the case described in Example \ref{cocom} the associated Hopf $*$-algebra is the group algebra of $\Gamma$.

The antipode and counit of $\cA$ may fail to have continuous extensions to $A$. We will return to the case when $e$ extends to a character of $A$ in Subsection \ref{contW}.

A quantum group $\GG=(A,\Delta)$ with the property that the antipode of $\cA$ extends to a continuous linear map $A\to{A}$ is called a quantum group of \emph{Kac type.} The situation when $\kappa$ has a continuous extension to $A$ can be characterized in many different ways. We will give some equivalent conditions for this phenomenon to hold in Subsection \ref{haar}.

\subsection{Haar measure}\label{haar}

\begin{theorem}[S.L.~Woronowicz]\label{h}
Let $\GG=(A,\Delta)$ be a compact quantum group. Then there exists a unique state $h$ on $A$ such that
\begin{equation*}
(\id\tens{h})\Delta(a)=(h\tens\id)\Delta(a)=h(a)\I
\end{equation*}
for all $a\in{A}$.
\end{theorem}

The state $h$ introduced in Theorem \ref{h} is called the \emph{Haar measure} of $\GG$. Clearly, in the case from Example \ref{com}, the state $h$ corresponds to integration with respect to the normalized Haar measure on $G$. In Example \ref{cocom} the Haar measure is the well known von Neumann trace. A more detailed analysis of the latter example shows that the Haar measure of a compact quantum group need not be a faithful state. Note that in the two examples of Haar measures we just discussed the state $h$ is a trace, i.e.~$h(ab)=h(ba)$ for all $a,b\in{A}$. One can ask if the Haar measure $h$ of a compact quantum group $\GG=(A,\Delta)$ is always a trace. This is not the case. The following theorem describes this situation.

\begin{theorem}[S.L.~Woronowicz]\label{KacT}
Let $\GG=(A,\Delta)$. Let $h$ be the Haar measure of $\GG$ and let $\kappa$ be the antipode of the Hopf $*$-algebra $\cA\subset{A}$. Then the following conditions are equivalent:
\begin{enumerate}
\item $\GG$ is of Kac type,
\item $h$ is a trace,
\item $\kappa$ has a bounded extension to $A$,
\item $\kappa^2=\id$,
\item $\kappa$ is a $*$-antiautomorphism of $\cA$.
\end{enumerate}
\end{theorem}

\subsection{Reduced quantum group}\label{red}

As we noted in Subsection \ref{haar} the Haar measure of a compact quantum group $\GG=(A,\Delta)$ may not be faithful. However the following theorem (essentially due to S.L.~Woronowicz) shows that we can always pass to a ``new version'' of $\GG$ which has a faithful Haar measure.

\begin{theorem}
Let $\GG=(A,\Delta)$ be a compact quantum group with Haar measure $h$. Then the left kernel of $h$
\[
\cJ=\bigl\{a\in{A}\st{h(a^*a)=0}\bigr\}
\]
is a two-sided ideal in $A$. Let $A_r$ be the quotient $A/\cJ$ and let $\rho_r:A\to{A_r}$ be the quotient map.
Then
\begin{enumerate}
\item there exists a unique $\Delta_r:A_r\to{A_r}\tens{A_r}$ such that the diagram
\[
\xymatrix
{
A\ar[d]_{\rho_r}\ar[rr]^-\Delta&&A\tens{A}\ar[d]^{\rho_r\tens\rho_r}\\
A_r\ar[rr]_-{\Delta_r}&&A_r\tens{A_r}
}
\]
is commutative,
\item $\GG_r=(A_r,\Delta_r)$ is a compact quantum group,
\item $\rho_r$ is injective on $\cA\subset{A}$ and $\rho_r(\cA)$ is the Hopf $*$-algebra associated to $\GG_r$. \end{enumerate}
\end{theorem}

Let $\GG=(A,\Delta)$ be a compact quantum group with associated Hopf $*$-algebra $\cA$. The compact quantum group $\GG_r$ is called the \emph{reduced version} of $\GG$. In the situation from Example \ref{com} we have $A_r=A$ because the Haar measure is faithful. On the other hand in case from Example \ref{cocom} where $A=\cst(\Gamma)$ for a discrete group $\Gamma$, we have $A_r=\cst_r(\Gamma)$ --- the reduced group $\cst$-algebra of $\Gamma$.

One can also adopt a different point of view and treat $A$ and $A_r$ as completions of $\cA$ with respect to different $\cst$-norms for which $\Delta_\cA$ is continuous.

\subsection{Universal quantum group}\label{univ}

\begin{theorem}\label{ut}
Let $\GG=(A,\Delta)$ be a compact quantum group with associated Hopf $*$-algebra $\cA$. There exists the enveloping $\cst$-algebra $A_u$ of $\cA$. Let $\rho_u:A_u\to{A}$ be the canonical epimorphism. Then
\begin{enumerate}
\item there exists a unique $\Delta_u:A_u\to{A_u}\tens{A_u}$ such that the diagram
\[
\xymatrix
{
A_u\ar[d]_{\rho_u}\ar[rr]^-{\Delta_u}&&A\tens{A}\ar[d]^{\rho_u\tens\rho_u}\\
A\ar[rr]_-{\Delta}&&A\tens{A}
}
\]
is commutative,
\item $\GG_u=(A_u,\Delta_u)$ is a compact quantum group,
\end{enumerate}
\end{theorem}

The compact quantum group $\GG_u$ described in Theorem \ref{ut} is called the \emph{universal version} of $\GG$. Since $A_u$ is defined as the completion of $\cA$ with respect to a maximal possible $\cst$-norm, it is clear that the Hopf $*$-algebra associated to $\GG_u$ is $\cA$.

If $A$ is commutative (i.e.~in the situation of Example \ref{com}) we always have $A_u=A=A_r$. Moreover if $A_r$ is commutative then so is $A$ and consequently $A_u=A=A_r$.

\subsection{Woronowicz characters}

Let $\GG=(A,\Delta)$ be a compact quantum group and let $\phi$ and $\psi$ be continuous functionals on $A$. Then the functional $(\phi\tens\psi)\comp\Delta$ is called the \emph{convolution} of $\phi$ and $\psi$ and is denoted by $\phi*\psi$. Similarly if $a\in{A}$ then we can define \emph{left} and \emph{right convolutions} $\phi*a$ and $a*\psi$ of $a$ with $\phi$ and $\psi$ respectively by
\[
\phi*a=(\id\tens\phi)\Delta(a)\quad\text{and}\quad{a}*\psi=(\psi\tens\id)\Delta(a).
\]
These definitions are straightforward generalizations of the notion of convolution of measures and measures and continuous functions on a compact group (cf.~Example \ref{com}).

If $\cA$ is the Hopf $*$-algebra associated to $\GG$ and $\phi$ and $\psi$ are linear functionals on $\cA$ then the we can use the same formulas to define $\phi*\psi$ and $\phi*a$, $a*\psi$ for $a\in\cA$.

\begin{theorem}[S.L.~Woronowicz]\label{slwc}
Let $\GG=(A,\Delta)$ be a compact quantum group with associated Hopf $*$-algebra $\cA$. Then there exists a unique family $(f_z)_{z\in\CC}$ of non-zero multiplicative functionals on $\cA$ such that
\begin{enumerate}
\item for any $a\in\cA$ the function $z\mapsto{f_z(a)}$ is entire holomorphic,
\item $f_0=e$ and $f_{z_1}*f_{z_2}=f_{z_1+z_2}$ for all $z_1,z_2\in\CC$,
\item\label{slwc3} for any $z\in\CC$ and $a\in\cA$
\[
f_z\bigl(\kappa(a)\bigr)=f_{-z}(a)\quad\text{and}\quad
f_{\Bar{z}}(a^*)=\Bar{f_{-z}(a)}
\]
\item for any $a\in\cA$ we have $\kappa^2(a)=f_{-1}*a*f_1$.
\end{enumerate}
\end{theorem}

It is clear from Theorem \ref{slwc}\eqref{slwc3} that $(f_{\mathrm{i}t})_{t\in\RR}$ is a one-parameter group (under the operation of convolution) of $*$-characters of $\cA$. They extend to characters of $A_u$ which we call \emph{Woronowicz characters.} The whole family $(f_z)_{z\in\CC}$ is the family of \emph{modular functionals} of $\GG$.

It is important to note that the family $(f_z)_{z\in\CC}$ may be trivial in the sense that $f_z=e$ for all $z$. Indeed it is always the case for Kac algebras:

\begin{theorem}\label{mKac}
Let $\GG=(A,\Delta)$ be a compact quantum group and $(f_z)_{z\in\CC}$ be the family of modular functionals of $\GG$. Then the following are equivalent:
\begin{enumerate}
\item $f_z=e$ for all $z$,
\item $\GG$ is of Kac type.
\end{enumerate}
\end{theorem}

Theorem \ref{mKac} implies, in particular, that commutative and cocommutative examples (Examples \ref{com} and \ref{cocom}) have trivial families of modular functionals.

\subsection{The modular group}

Let $\GG=(A,\Delta)$ be a compact quantum group with associated Hopf $*$-algebra $\cA$ and modular functionals $(f_z)_{z\in\CC}$. The formula
\[
\sigma_t(a)=f_{\mathrm{i}t}*a*f_{\mathrm{i}t}
\]
for $t\in\RR$ and $a\in\cA$ defines a one parameter group of automorphisms of $\cA$. This group is called the \emph{modular group} of $\GG$. Let us note that the modular group has a continuous extension to $A_u$, because all Woronowicz characters extend to characters of $A_u$ and it is easy to see that the family $(f_{\mathrm{i}t})_{t\in\RR}$ is continuous on $A_u$. A theorem of S.L.~Woronowicz (\cite[Theorem 2.6]{cqg}) asserts that the one parameter group $(\sigma_t)_{t\in\RR}$ is also continuous on $A_r$ and is intimately connected with the failure of the Haar measure $h$ of $\GG$ to be a trace ($h$ is a $\sigma$-KMS state).

\section{Some tools}\label{tools}

\subsection{Continuity of Woronowicz characters}\label{contW}

Let $\GG=(A,\Delta)$ be a compact quantum group and let $G$ be the set of non-zero multiplicative functionals on $A$. Clearly $G$ is a weak$^*$ compact subset of the unit sphere in $A^*$ and convolution of functionals defines on $G$ a structure of a compact associative semigroup. There are many ways to see that $G$ is in fact a compact group. Indeed, let $\gamma:A\to\C(G)$ be the Gelfand transform. Clearly $\gamma$ is a surjective $*$-homomorphism and if $\Delta_G$ is the map $\C(G)\to\C(G\times{G})$ given by
\[
(\Delta_G(f))(\phi,\psi)=f(\phi*\psi)
\]
then $(\gamma\tens\gamma)\comp\Delta=\Delta_G\comp\gamma$ by the very definition of convolution product (in other words $\gamma$ is a \emph{quantum group morphism}). Therefore we have the density of linear spans of
\[
\bigl\{\Delta_G(f)(\I\tens{g})\st{f,g}\in\C(G)\bigr\}
\quad\text{and}\quad\bigl\{(f\tens\I)\Delta(g)\st{f,g}\in\C(G)\bigr\}
\]
in $\C(G)\tens\C(G)$ which is equivalent to cancellation laws in $G$ (cf.~Example \ref{com}). It follows that $G$ is a compact group. In particular it has the unit element $\phi_0$.

The next thing we want to check is that $\phi_0$ is the extension to $A$ of the counit $e$ defined on the Hopf $*$-algebra $\cA$ associated with $\GG$. This follows by considering a finite dimensional unitary representation $u\in{M}_n(\CC)\tens{A}$ and the unitary matrix $U=(\id\tens\phi_0)u$. It follows from \eqref{Delu} that $U=U^2$, which for a unitary matrix means $U=\I$. This means that $\phi_0(u_{i,j})=\delta_{i,j}$ for any finite dimensional unitary representation $u=(u_{i,j})$ of $\GG$. The counit $e$ of $\cA$ has the same values on matrix elements of representations, and $\cA$ is spanned by these matrix elements. Therefore $\phi_0=e$ on $\cA$.

\begin{theorem}\label{tool}
Let $\GG=(A,\Delta)$ be a compact quantum group such that the $\cst$-algebra $A_r$ possesses a character. Then all Woronowicz characters are continuous on $A$.
\end{theorem}

\begin{proof}
From the discussion preceding statement of the theorem we know that the nonempty set of characters of $A_r$ is a compact group whose unit is the extension to $A_r$ of the counit $e$ of the Hopf $*$-algebra $\cA$ associated with $\GG_r$. We will still write $e$ to denote this extension.

Now for any $a\in\cA$ and $t\in\RR$ we have
\[
\begin{split}
e\bigl(\sigma_t(a)\bigr)
&=(f_{\mathrm{i}t}\tens{e}\tens{f_{\mathrm{i}t}})(\Delta_\cA\tens\id)\Delta_\cA(a)\\
&=(f_{\mathrm{i}t}\tens{f_{\mathrm{i}t}})\Delta_\cA(a)
=(f_{\mathrm{i}t}*{f_{\mathrm{i}t}})(a)=f_{2\mathrm{i}t}(a)
\end{split}
\]
and we know that both $e$ and $(\sigma_t)_{t\in\RR}$ are continuous on $A_r$. This means that Woronowicz characters $(f_{\mathrm{i}t})_{t\in\RR}$ also have continuous extensions to $A_r$. Since $A_r$ is the image of $A$ under $\rho_r$ (cf.~Subsection \ref{red}) which is the identity on $\cA$, all Woronowicz characters also extend to characters of $A$.
\end{proof}

It has to be noted that in fact a statement much stronger than Theorem \ref{tool} is true. Namely, existence of a character on $A_r$ implies that $\rho_u$ and $\rho_r$ are isomorphisms, so in particular, $A_r=A=A_u$ (\cite{bmt}) and Woronowicz characters must extend to $A$. However, we will make use solely of the weaker statement presented above.

\subsection{Other results}

We will need two more results that will help us disprove existence of compact quantum group structure on some compact quantum spaces. The first one follows from a very deep theorem of B\'edos, Murphy and Tuset ({\cite[Theorem 1.1]{bmt2}}).

\begin{theorem}\label{bmtnuc}
Let $\GG=(A,\Delta)$ be a compact quantum group of Kac type with associated Hopf $*$-algebra $\cA$. Then the following are equivalent:
\begin{enumerate}
\item the $\cst$-algebra $A_r$ is nuclear,
\item the counit of $\cA$ is continuous on $A_r$.
\end{enumerate}
\end{theorem}

Recall that a $\cst$-algebra $A$ is \emph{nuclear} if for any $\cst$-algebra $B$ the algebraic tensor product $A\atens{B}$ admits a unique $\cst$-norm. This property can be also characterized in many different ways, but all we will need is the fact that a crossed product of a commutative $\cst$-algebra by an action of a commutative group is nuclear (\cite[Proposition 2.1.2]{roerdam}).

The next fact we will need is the following:

\begin{theorem}[{\cite[Remark A.2]{qbc}}]\label{qbcT}
Let $\GG=(A,\Delta)$ be a compact quantum group such that $A$ admits a faithful family of tracial states. Then $\GG$ is of Kac type.
\end{theorem}

\section{Examples}\label{ex}

In this section we will give examples of compact quantum spaces which do not admit any compact quantum group structure. In each case we will use the results collected in Section \ref{tools} to prove non-existence of such a structure. The quantum spaces under consideration are all discussed in the survey article \cite{dab}, where further references can be found.

\begin{example}[The quantum torus]\label{qt}
Let us fix $\theta\in]0,1[$. The \emph{quantum two-torus} $\TT_\theta^2$ is the quantum space corresponding to the rotation $\cst$-algebra $A_\theta$, i.e.~the universal $\cst$-algebra generated by two unitary elements $u$ and $v$ satisfying the relation $uv=\mathrm{e}^{2\pi\mathrm{i}\theta}vu$ (\cite{rie1}). This $\cst$-algebra is nuclear (indeed, $A_\theta=\C(\TT)\rtimes\ZZ$, where the action is by rotation by $2\pi\theta$). Moreover it possesses a faithful trace (unique if $\theta$ is irrational).

Let us assume that there exists $\Delta:A_\theta\to{A_\theta}\tens{A_\theta}$ such that $\GG=(A_\theta,\Delta)$ were a compact quantum group. We know from theorem \ref{qbcT} that the Haar measure of $\GG$ must then be a trace and $\GG$ is of Kac type. Moreover, since $A_\theta$ is nuclear, we know from Theorem \ref{bmtnuc} that $(A_\theta)_r$ must have a character (because then $(A_\theta)_r$ is also nuclear). But if this were the case then $A_\theta$ would admit a character, since $(A_\theta)_r$ is a quotient of $A_\theta$ (in fact, for irrational $\theta$ the algebra $A_\theta$ is simple, so there are no proper quotients of $A_\theta$) and we know that $A_\theta$ does not admit any characters.

The same reasoning can be applied to show that higher dimensional quantum tori (\cite{rie2}) do not admit a compact quantum group structure for nontrivial deformation parameters. If the deformation parameters are trivial (that means $\theta=0$ for the two-torus) the resulting $\cst$-algebra $A_\theta$ is just the algebra of continuous function on a torus and, as such, carries a compact quantum group structure.

It should be noted that the quantum two-torus can be made into a ``part'' of a compact quantum group. This was done by P.M.~Hajac and T.~Masuda in \cite{qdt}.

There is one more interesting remark. The theory of quantum groups on operator algebra level has its version in the language of von Neumann algebras (see e.g.~\cite{kvvn}). The passage to von Neumann algebras is achieved by taking weak closure in the GNS representation defined by the Haar measure. If $\GG=(A_\theta,\Delta)$ were a compact quantum group for some irrational $\theta$ then the unique tracial state of $A_\theta$ would be its Haar measure (as we said earlier this follows from Theorem \ref{qbcT} and Theorem \ref{KacT}). Therefore the resulting von Neumann algebra would be the hyperfinite factor of type $\mathrm{II}_1$. We already know that the quantum torus is not a quantum group. However the ``algebra of measurable essentially bounded functions'' on the quantum torus is the same as that of ``measurable functions'' on many compact quantum groups. Indeed we get the same von Neumann algebra when we start with $\cst(\Gamma)$ for any amenable i.c.c.~discrete group $\Gamma$.
\end{example}

\begin{example}[Batteli-Elliott-Evans-Kishimoto quantum two-spheres]
The algebra of continuous functions on a Bratteli-Elliott-Evans-Kishimoto quantum two-sphere is by definition the fixed point subalgebra $C_\theta$ of $A_\theta$ considered in Example \ref{qt} under the action of $\ZZ_2$ sending $u$ and $v$ to their adjoints (\cite{beek,dab}). For $\theta=0$ we have $C_\theta=\C(\sS^2)$ and this $\cst$-algebra carries no compact quantum group structure because $\sS^2$ is not a topological group. For $\theta\in]0,1[$ the argument used to show that $\TT_\theta^2$ is not a compact quantum group works perfectly well for the BEEK quantum two-spheres. All we need is to know that $C_\theta$ is a nuclear $\cst$-algebra with a faithful trace which admits no characters.
\end{example}

\begin{example}[Standard Podle\'s quantum sphere]\label{expod}
In \cite{spheres} Piotr Podle\'s defined and studied a class of quantum spaces which later came to be known as Podle\'s (quantum) spheres. These are compact quantum spaces endowed with an action of the quantum $\mathrm{SU}(2)$ group (\cite{su2}) mimicking the standard action of $\mathrm{SU}(2)$ on $\sS^2$. Podle\'s found all such objects and showed that they form a family $(\sS^2_{q,c})$ labeled by a parameter $c\in[0,\infty]$ (the parameter $q$ is related to the particular quantum $\mathrm{SU}(2)$ group). Only for $c=0$ can $\sS^2_{q,c}$ be considered as a quotient homogeneous space the quantum $\mathrm{SU}(2)$ and this quantum sphere is referred to as the \emph{standard} Podle\'s quantum sphere (cf.~\cite{podsym}).

The algebra of continuous functions on $\sS^2_{q,0}$ is isomorphic to $\cK^+$, i.e.~the minimal unitization of the algebra of compact operators on a separable Hilbert space. This $\cst$-algebra has a unique proper ideal which is maximal. This makes it easy to see that if there existed $\Delta:\cK^+\to\cK^+\tens\cK^+$ such that $\GG=(\cK^+,\Delta)$ were a compact quantum group then the Haar measure $h$ of $\GG$ would have to be faithful. Otherwise the left kernel $\cJ$ of $h$ would be either $\{0\}$ or so big, that the quotient $A_r=\cK^+/\cJ$ would be commutative (equal to $\CC$). But we already said in Subsection \ref{univ} that compact quantum groups described by commutative $\cst$-algebras are automatically universal. It follows that $\GG$ must be equal to its reduced version. Note that on $\cK^+$ there exists a nontrivial character.

There are no faithful traces on $\cK^+$, so that the family of modular functionals must be nontrivial. But since $\cK^+$ has a character, the Woronowicz characters must be continuous and nontrivial on $\cK^+$ (Theorem \ref{tool}). However there is only one character on $\cK^+$ which shows that existence of $\Delta$ is impossible.
\end{example}

\begin{example}[Natsume-Olsen quantum two-spheres]
In \cite{natsume} a family of quantum spaces was introduced which we call the family of Natsume-Olsen quantum two-spheres. The family is parametrized by a parameter $t\in\bigl[0,\tfrac{1}{2}\bigr[$. The $\cst$-algebra $B_t$ corresponding to $t$ is the universal $\cst$-algebra generated by two elements $z$ and $\zeta$ satisfying the following relations:
\begin{eqnarray}
\zeta^*\zeta+z^2=\I=\zeta\zeta^*+(t\zeta\zeta^*+z)^2,\nonumber\\
\zeta{z}-z\zeta=t\zeta(\I-z^2).\qquad\;\quad\nonumber
\end{eqnarray}
For $t=0$ we have $B_t=\C(\sS^2)$ and $\sS^2$ is not a compact group. 

In case $t>0$ let us suppose that $\GG=(B_t,\Delta)$ is a compact quantum group for some comultiplication $\Delta:B_t\to{B_t}\tens{B_t}$. Then we note that $B_t$ has an ideal $\cI$ isomorphic to $\C(\TT)\tens\cK$ such that $B_t/\cI=\CC^2$ (\cite{natsume}). This can be used to show that the Haar measure $h$ of $\GG$ cannot be a trace. Indeed, if it were then a simple argument shows that the left kernel $\cJ$ of $h$ would contain $\cI$. Then the quotient $(B_t)_r=B_t/\cJ$ would be commutative and we would arrive at a contradiction as e.g.~in Example \ref{expod}. Then one must do a little work to show that $(B_t)_r$ admits a character (\cite{non}). By Theorem \ref{tool} all Woronowicz characters must be continuous on $B_t$ and by Theorems \ref{mKac} and \ref{KacT} and the fact that $h$ is not a trace we see that $B_t$ must have a nontrivial one-parameter continuous family of characters. This contradicts a simple fact that the character space of $B_t$ consists of two points. It follows that Natsume-Olsen quantum spheres do not admit a compact quantum group structure.
\end{example}

\section{On the quantum disk}

Let us now describe some partial results on the question whether the \emph{quantum disk} could have a quantum group structure. The quantum disk is the compact quantum space described by the Toeplitz algebra $\cT$ (\cite{kl}). We do not have a proof that the quantum disk does not admit a compact quantum group structure, but we can go quite far along the lines of our previous examples.

If we assume that $\GG=(\cT,\Delta)$ is a compact quantum group then the Haar measure $h$ of $\GG$ must be faithful. This is because $\cK$ is an essential ideal of $\cT$, and the left kernel $\cJ$ of $h$ would have a non zero intersection with $\cK$ which by simplicity of $\cK$ would be all of $\cK$. In particular $\cJ$ would contain $\cK$ and the quotient would be commutative. As we stated in Subsection \ref{univ} this would imply that the $\cst$-algebra $\cT$ is commutative. In particular $h$ cannot be a trace since there is no faithful trace on $\cT$.

As $h$ is faithful, we see that $\GG=\GG_r$ and since $\cT$ admits a character, all Woronowicz characters must be continuous on $\cT$. The character space of the $\cst$-algebra $\cT$ is homeomorphic to $\TT$ and carries a structure of a compact group. Let us first see that this is exactly the set of Woronowicz characters with the group structure of quotient of $\RR$ (recall that $f_{it}*f_{is}=f_{i(t+s)}$). Indeed, let $G$ be the compact group of characters of $\cT$ and let $G_W$ be the subgroup consisting of Woronowicz characters. $G_W$ is nontrivial (because $h$ is not tracial) and it is a connected subset of $G$. Therefore $G_W$ contains the unit of $G$ as an interior point (cf.~Section \ref{intro}) and thus a neighborhood of the unit of $G$. But $G$ is connected, so it is algebraically generated by any open neighborhood of its unit. It follows that $G_W=G$.

The Gelfand transform $\gamma:\cT\to\C(\TT)$ is a morphism of compact quantum groups, where the comultiplication on $\C(\TT)$ comes from group structure of $\TT=G=G_W$. Using some known results about Toeplitz algebras (\cite{douglas}) one can show that there is an isometry $\bx$ which generates $\cT$ and is mapped to a group like generator $\bz$ of $\C(\TT)$ under the Gelfand transform (this means that the comultiplication on $\C(\TT)$ sends $\bz$ to $\bz\tens\bz$). It follows that $\Delta(\bx)=\bx\tens\bx+X$, where $X\in\ker\gamma\tens\gamma=\cT\tens\cK+\cK\tens\cT$. Now, if $X$ could be shown to be zero we would be done, since in that case $\bx$ would be group like and by \cite[Theorem 2.6(2)]{cqg} would belong to the Hopf $*$-algebra associated to $\GG$. However group like elements of Hopf algebras must be invertible and $\bx$ is not. Unfortunately, although we can show a number of properties of $X$, as for now the property that $X=0$ is not within our reach. In fact, using techniques from the theory of Hopf-Galois extensions, one could show that the weaker property that $(\id\tens\gamma)(X)=0$ would be enough to disprove existence of a compact quantum group structure on the quantum disk.

\section*{Acknowledgments}

Research partially supported by Polish government grant no.~N201 1770 33. This paper is based on a lecture delivered at the 19$^\textrm{th}$ International Conference on Banach Algebras held at Będlewo, July 17--27, 2009. The support for the meeting by the Polish Academy of Sciences, the European Science Foundation, and the Faculty of Mathematics and Computer Science of the Adam Mickiewicz University at Pozna\'n is gratefully acknowledged.

\end{document}